\documentclass[letterpaper, 10 pt, conference]{ieeeconf}  

\IEEEoverridecommandlockouts                              
\usepackage[OT1]{fontenc} 
\usepackage[english]{babel} 
\usepackage[driverfallback=dvipdfmx,hidelinks]{hyperref}
\usepackage{amsmath,amssymb,mathrsfs}
\usepackage{cleveref}
\usepackage{bm}
\usepackage{array}
\usepackage{multirow}
\usepackage[hang]{footmisc}

\usepackage{booktabs}
\usepackage{algorithm}
\usepackage{algpseudocode}
\usepackage{comment}
\usepackage{subfigure}
\usepackage{tabularx}
\usepackage{bibentry}
\usepackage[noadjust]{cite}
\usepackage{bbm}
\usepackage{menukeys}
\usepackage[short]{optidef}

\let\labelindent\relax
\usepackage{enumitem}

\usepackage{amsthm}
\newtheorem{problem}{Problem}

\newtheorem{lemma}{Lemma}
\newtheorem{remark}{Remark}

\newcommand{\norm}[1]{\left\lVert#1\right\rVert}

\newcommand{\R}{\mathbb{R}}

\newcommand{\diff}{\mathrm{d}}

\crefname{align}{Eq.}{Eqs.}
\crefname{equation}{Eq.}{Eqs.}
\crefname{figure}{Fig.}{Figs.}
\crefname{table}{Table}{Tables}
\crefname{theorem}{Theorem}{Theorems}
\crefname{definition}{Definition}{Definitions}
\crefname{lemma}{Lemma}{Lemmas}
\crefname{assumption}{Assumption}{Assumptions}
\crefname{proof}{Proof}{Proofs}
\crefname{remark}{Remark}{Remarks}
\crefname{problem}{Problem}{Problems}
\crefname{proposition}{Proposition}{Propositions}
\crefname{corollary}{Corollary}{Corollaries}
\crefname{section}{Section}{Sections}

\usepackage[numbered,framed]{matlab-prettifier}
\usepackage{filecontents}

\graphicspath{ {./img/} }

\title{
Smooth Indirect Solution Method for
State-constrained Optimal Control Problems
with Nonlinear Control-affine Systems
}

\author{Kenshiro Oguri 
\thanks{K.~Oguri is with the School of Aeronautics and Astronautics, Purdue University, IN 47907, USA ({\tt koguri@purdue.edu}).
}}

\usepackage{tikz}
\usepackage{textcomp}
\newcommand\copyrighttext{%
  \footnotesize \textcopyright 2023 IEEE. Personal use of this material is permitted.
  Permission from IEEE must be obtained for all other uses, in any current or future
  media, including reprinting/republishing this material for advertising or promotional
  purposes, creating new collective works, for resale or redistribution to servers or
  lists, or reuse of any copyrighted component of this work in other works.
  DOI: \underline{\href{https://ieeexplore.ieee.org/document/10383623}{10.1109/CDC49753.2023.10383623}}.
  Presented at 2023 IEEE CDC.
  }
\newcommand\copyrightnotice{%
\begin{tikzpicture}[remember picture,overlay]
\node[anchor=south,yshift=10pt] at (current page.south) {\fbox{\parbox{\dimexpr\textwidth-\fboxsep-\fboxrule\relax}{\copyrighttext}}};
\end{tikzpicture}%
}

\begin{document}

\maketitle
\thispagestyle{empty}
\pagestyle{empty}
\copyrightnotice

\begin{abstract}
This paper presents an indirect solution method for state-constrained optimal control problems to address the long-standing issue of discontinuous control and costate under state constraints.
It is known in optimal control theory that a state inequality constraint introduces discontinuities in control and costate, rendering the classical indirect solution methods ineffective.
This study re-examines the necessary conditions of optimality for a class of state-constrained optimal control problems, and shows the uniqueness of the optimal control that minimizes the Hamiltonian under state constraints, which leads to a unifying form of the necessary conditions.
The unified form of the necessary conditions opens the door to addressing the issue of discontinuities in control and costate by modeling them via smooth functions.
This paper exploits this insight to transform the originally discontinuous problems to smooth two-point boundary value problems that can be solved by existing nonlinear root-finding algorithms.
This paper also shows the formulated solution method to have an anytime algorithm-like property, and then numerically demonstrates the solution method by an optimal orbit control problem.
\end{abstract}


\section{Introduction}
\label{sec:intro}

This paper presents a new indirect solution method for state-constrained optimal control problems to address the long-standing issue of discontinuous costate and control under state constraints.
It is widely known that imposing pure state path constraints, $S(x, t) \leq0 $, in optimal control problems introduces discontinuities in control and costate \cite{Hartl1995,Bryson1975,Speyer1968}.
Despite such difficulties, state-constrained optimal control problems arise in many fields, including aerospace, robotics, and other engineering and science applications.

In practice, many state-constrained optimal control problems are solved via direct methods.
Direct methods convert continuous-time optimal control problems into finite-dimensional problems by parameterizing the problems via multiple shooting \cite{Bock1984}, numerical collocation \cite{Hargraves1987a}, or exact time discretization \cite{Szmuk2020}.
The resulting parameter optimization problems are solved via nonlinear programming \cite{Nocedal2006}, sequential convex programming \cite{Mao2016}, model predictive control \cite{Camacho2007}, etc.
Many successful commercial software for solving general optimal control problems, such as GPOPS \cite{Patterson2014a}, fall under the category of direct methods.
To incorporate inequality constraints in the direct method framework, they almost always employ the penalty function method and its variants \cite{DENHAM1964,Teo1989}, which impose soft constraints by penalizing the violation of constraints and increase the penalty weight over iterations, asymptotically achieving the feasibility and optimality.
These methods are usually combined with techniques in numerical optimization, such as the interior point method \cite{Nocedal2006} and the augmented Lagrangian method \cite{Bertsekas1982}.

Although direct methods provide flexibility in solving a variety of problems, they often compromise the feasibility and/or optimality due to the nature of the penalty function method \cite{DENHAM1964,Teo1989} and problem parameterization, which often involves approximation \cite{Bock1984,Hargraves1987a,Camacho2007}.
The computational complexity can be also an issue due to the large number (hundreds to thousands) of variables introduced by the parameterization \cite{Bock1984,Hargraves1987a,Mao2016}.
In contrast, indirect methods possess an advantage of the fewer dimensionality of the problem (${<}$10 unknowns).
The fewer dimensionality is attractive particularly for complex control problems, such as those involving uncertainty \cite{Oguri2022c} and multiple phases \cite{Sidhoum2023}.

Motivated by these facts, this paper revisits the use of indirect methods.
However,
classical indirect methods struggle to solve state-constrained problems.
Due to the discontinuities in control and costate resulting from Pontryagin's minimum principle \cite{Bryson1975}, classical approaches divide the entire trajectory into sub-arcs and solve each sub-arc as separate two-point boundary value problems (TPBVPs) while satisfying the transversality conditions \cite{Hartl1995,Bryson1975,Speyer1968}.
This approach requires \textit{a priori} information about the sub-arc structure (e.g., number of constrained arcs), which is usually unknown in practice.

To address the issue, some recent studies have developed indirect solution methods with various approaches, such as
using saturation function for problems with lower/upper bounds on state \cite{Graichen2010};
introducing slack variables and penalty function \cite{Fabien2014};
deriving a variant of the minimum principle by requiring the continuity of Lagrange multipliers associated with state constraints \cite{Arutyunov2015,Chertovskih2021}.
On the other hand, like any approaches, they have their own assumptions and limitations (e.g., type of constraints considered, continuity).

This study tackles state-constrained optimal control via a different indirect method approach that explicitly addresses the discontinuities arising from the minimum principle.
First, we revisit necessary conditions of optimality for constrained problems and show the uniqueness of optimal control that minimizes the constrained control Hamiltonian.
This analysis leads to a unifying form of optimality necessary conditions, which enables us to 
address the issue of discontinuities in control and costate by modeling them via smooth functions with a sharpness parameter $\rho$.
This approach transforms the originally discontinuous problem to a smooth TPBVP that can be solved by existing root-finding algorithms, where the sharpness parameter $\rho$ controls the optimality of the resulting solutions.
It is also shown that intermediate solutions of the continuation process respect the state constraint, which is often a desirable property for safety-critical systems.


\section{Preliminary}

\subsection{Problem Statement}
\label{sec:problem}

Consider a nonlinear, control-affine dynamical system:
\begin{align}
\dot{x} = f(x,u,t) = f_0(x, t) + F(x,t) u 
\label{eq:dynamics}
\end{align}
where $x\in\R^{n} $, $u\in\R^{m} $, and $t\in\R$ are the state, control, and time, respectively;
this system arises in many aerospace \cite{Conway2010,Oguri2019b} and robotic systems \cite{Bonalli2022a}.
$F(\cdot)$ is twice continuously differentiable, which holds for many practical problems (e.g., for most robotic/aerospace problems, $F = [0_{m\times m}, I_m]^\top $).

Without loss of generality, consider the cost functional in the Lagrange form%
\footnote{The other forms of cost, the Mayer and Bolza forms, can be converted to the Lagrange form.
Also, a time interval $\tau\in [\tau_0, \tau_f]$ can be transformed to $t\in[0, T] $ via $t = \tau - \tau_0$ and $T = \tau_f - \tau_0 $.}:
\begin{align}
J = \int_{0}^{T} L(x,u,t)\diff t.
\label{eq:originalObjective}
\end{align}
A feasible trajectory must satisfy terminal constraints
\begin{align}
\psi(x(0), x(T), T) = 0
\label{eq:terminalConst}
\end{align}
where $\psi: \R^n\times \R^n\times \R \mapsto \R^q $,
and a scalar state constraint
\begin{align}
S(x, t) \leq 0.
\label{eq:pathConst}
\end{align}
$S$ is assumed to be twice continuously differentiable and first order, implying that the first time derivative of $S$ contains the control term $u$ explicitly, i.e.,
\begin{align}
\begin{aligned}
S^{(1)}(x, u, t)\triangleq 
\dot{S} = S_x f + S_t 
= S_x f_0 + S_xFu + S_t
\label{eq:S1}
\end{aligned}
\end{align}
where $S_x\in\R^{1\times n}, S_t\in\R $ are the gradients of $S$ with respect to $x$ and $t$, i.e., $S_x \triangleq \nabla_x S $ and $S_t \triangleq \nabla_t S $.
The first-order constraint assumption is thus equivalent to $S_x F \neq 0,\ \forall x,t $.
Note that this assumption may limit the applicability of the proposed method, as $S_x F = 0$ for some practical constraints.

Thus, our original problem is formulated as in \cref{prob:originalOCP}.
\begin{problem}[Original problem]
\label{prob:originalOCP}
Find $x^{*}(t), u^{*}(t), T^*$ that minimize the cost \cref{eq:originalObjective} subject to the dynamics \cref{eq:dynamics}, terminal constraints \cref{eq:terminalConst}, and state constraint \cref{eq:pathConst}.
\end{problem}

Although the analysis in this study is focused on a scalar state constraint $S$ for conciseness, generalization to a vector-valued state constraint $S$ is straightforward by requiring the following condition in addition to $S_x F \neq 0,\ \forall x,t $;
for $S\in\R^{l}$ with $l>1$, rows of $S_x F$ are independent (similar to the linear independence constraint qualification, or LICQ, which is often assumed in nonlinear programming algorithms \cite{Nocedal2006}).

\subsection{Optimality Necessary Conditions}
\label{sec:necessaryCondition}

The necessary conditions of optimality take different forms depending on whether the trajectory is on a constrained arc ($S=0 $) or on an unconstrained arc ($S<0 $) \cite{Bryson1975}.
We use $u_o^* $ to refer to the optimal control on an unconstrained arc while $u_c^* $ to that on a constrained arc.
$u^*$ collectively denotes the optimal control regardless of constrained/unconstrained arc.

Clearly, a trajectory is characterized as follows:
\begin{align} \begin{aligned}
\begin{cases}
\mathrm{unconstrained\ arc} 	& \mathrm{if\ } S < 0 \\
\mathrm{constrained\ arc} 		& \mathrm{if\ } S = 0\ \land\ S^{(1)}_o \geq 0 \\
\mathrm{infeasible} 			& \mathrm{if\ } S > 0 
\end{cases}
\label{eq:trajectoryClasification}
\end{aligned} \end{align}
where $S^{(1)}_o $ is a shorthand notation of $S^{(1)}(\cdot) $ under $u_o^* $, i.e., $S^{(1)}_o \triangleq S^{(1)}(x, u_o^*, t) $.
When $S = 0$ and $S^{(1)}_o \leq 0$ at $t=t_i$, the trajectory is leaving the constrained arc, called an exit point.
Likewise, $S = 0$ and $S^{(1)}_o > 0$ at an entry point.

\subsubsection{Unconstrained arc}
\label{sec:NCunconst}

On an unconstrained arc, from Pontryagin's minimum principle, $u_o^*$ must satisfy
\begin{align}
u_o^* \in \arg\min_{u} H_o
\label{eq:unconstOptimalControl}
\end{align}
$H_o $ is the unconstrained control Hamiltonian, defined as:
\begin{align} \begin{aligned}
H_o(x,u,\lambda,t) = L(x,u,t) + \lambda^\top [f_0(x,t) + F(x,t)u]
\label{eq:Hamiltonian}
\end{aligned} \end{align}
where $\lambda\in\R^{n}$ is the Lagrange multiplier associated with \cref{eq:dynamics}, or costate.
The costate dynamics are given by
\begin{align}
- \dot{\lambda}^\top   
= \nabla_x H_o
= \nabla_x L + \lambda^\top \nabla_x f
\label{eq:unconstainedCostateDynamics}
\end{align}

\subsubsection{Constrained arc}
\label{sec:NCconst}

On a constrained arc, a feasible control must satisfy $S^{(1)}(x, u, t) \leq 0$.
Among several approaches in literature \cite{Hartl1995,Bryson1975}, this study takes an approach known as \textit{indirect adjoining} to derive necessary conditions.%
 
The indirect adjoining approach adjoins the control-dependent constraint $S^{(1)}(x,u,t) \leq 0 $ in addition to \cref{eq:dynamics,eq:terminalConst} to yield the control Hamiltonian \cite{Bryson1975}:
\begin{align} \begin{aligned}
H_c(x,u,\lambda,\mu,t) = L + \lambda^\top f + \mu S^{(1)}
\label{eq:HamiltonianConst}
\end{aligned} \end{align}
where $H_c$ is the Hamiltonian on a constrained arc;
$\mu (\geq 0) $ is a Lagrange multiplier associated with $S^{(1)} \leq 0$.
Denoting the optimal multiplier by $\mu^*$, the following must hold \cite{Bryson1975}:
\begin{align}
(u_c^*, \mu^*)     \in\{u, \mu \geq0 \mid \min_{u,\mu} H_c\ \text{s.t.\ } S^{(1)} = 0 \}
\label{eq:constrainedOptimalControl}
\end{align}
The costate dynamics are given by
\begin{align*}
- \dot{\lambda}^\top   
= 
\nabla_x H_c(x, u, \lambda, \mu, t) 
= \nabla_x L + \lambda^\top \nabla_x f + \mu \nabla_x S^{(1)}
\end{align*}


\subsubsection{Discontinuity at corners}
\label{sec:jumpConditionDet}

It is known that $\mu^*$ may be discontinuous at the time of entry to a constrained arc, called a \textit{corner}%
\footnote{In fact, one may also choose the discontinuity to occur at the exit corner or distribute at both of the entry and exit corners (e.g., \cite{Hartl1995}). This paper chooses to have the corner discontinuity at entry corners.}.
At a corner (say $t=t_i$), $\lambda$ and $H$ obey \cite{Bryson1975,McINTYRE1967}:
\begin{align} \begin{aligned}
\Delta \lambda_i &\triangleq 
\lambda(t_i^+) - \lambda(t_i^-)
= - \mu(t_i^+) [S_x(t_i)]^\top,\\
\Delta H_i &\triangleq 
H(t_i^+) - H(t_i^-)
= \mu(t_i^+) S_t(t_i)
\label{eq:costateJump}
\end{aligned} \end{align}


\section{Uniqueness of Optimal Control and Multiplier}
\label{sec:uniqueness}

Although \cref{eq:constrainedOptimalControl} states necessary conditions for $(u^*, \mu^*) $ to satisfy, it does not clarify whether such $(u^*, \mu^*) $ is unique.
The uniqueness of $(u^*, \mu^*) $ is crucial for numerically solving optimal control problems via indirect methods.
Numerical solutions for problems with no such uniqueness guarantee typically have to rely on direct methods (e.g., GPOPS \cite{Patterson2014a}) with greater computational complexity.

\subsection{Optimal Control on Unconstrained Arcs}
\label{sec:assumption}
We assume that the unique solution to \cref{eq:unconstOptimalControl} is available in a closed form and continuously differentiable by satisfying the sufficient condition for the minimum Hamiltonian (known as the Legendre-Clebsch condition \cite{Bryson1975}):
\begin{align}
\begin{aligned}
&\nabla_u H_o = \nabla_u L(x, u, t) + \lambda^\top F = 0\quad  \land \\ 
&\nabla^2_{uu} H_o = \nabla_{uu}^2 L(x, u, t) \succ 0
\end{aligned}
\label{eq:Legendre-Clebsch}
\end{align}
Note that this assumption holds for many problems that are solved by indirect methods, as demonstrated in \cref{sec:analyticalExample}.
We generically express $u_o^*$ that satisfies \cref{eq:Legendre-Clebsch} as:
\begin{align}
u_o^* = g(x,\lambda,t)
\label{eq:unconstSolution}
\end{align}

\subsection{Optimal Control on Constrained Arcs}
\label{sec:constrainedArc}

It can be shown that, with the assumption made in \cref{sec:assumption}, \cref{prob:originalOCP} has the unique pair of $(u_c^*, \mu^*)$ $\forall t$ on constrained arcs.
\cref{lemma:uniqueConstrainedSolution} formally states this fact.

\begin{lemma}
\label{lemma:uniqueConstrainedSolution}
Under the assumption that the unique solution to \cref{eq:unconstOptimalControl} is available in a closed form \cref{eq:unconstSolution}, the solution to \cref{eq:constrainedOptimalControl} is uniquely determined in the following form:
\begin{align}
u_c^* =& g(x,\lambda + \mu S_x^\top,t),
\label{eq:constSolution}
\\
\mu^* =& h(x,\lambda,t) \geq 0,
\label{eq:constSolution-mu}
\end{align}
where $h$ is continuously differentiable.
\end{lemma}

\begin{proof}
The goal is to find the unique pair of $u_c^*$ and $\mu*$ that solves the following Hamiltonian minimization:
\begin{align}
(u_c^*, \mu^* ) = \arg \min_{\mu\geq 0} \min_u H_c&\ \text{s.t.\ } S^{(1)} = 0
\label{eq:constrainedHamiltonianCondition}
\end{align}
where $H_c$ is given by \cref{eq:HamiltonianConst}.
Expanding $H_c$, we have
$H_c = L + (\lambda^\top + \mu S_x)(f_0 + Fu) + \mu S_t$.
The Legendre-Clebsch condition for $H_c$ leads to
\begin{subequations}
\begin{align}
&\nabla_u H_c = \nabla_u L(x, u, t) + (\lambda^\top + \mu S_x)F = 0
\label{eq:Legendre-Clebsch1}
\quad \land\\
&\nabla^2_{uu} H_c = \nabla_{uu}^2 L(x, u, t) \succ 0
\label{eq:Legendre-Clebsch2}
\end{align}
\label{eq:Legendre-Clebsch12}
\end{subequations}
Noting \cref{eq:Legendre-Clebsch}, it is clear that $u_c^*$ that satisfies \cref{eq:Legendre-Clebsch12} is given by \cref{eq:constSolution}.
Next, note that \cref{eq:Legendre-Clebsch1,eq:Legendre-Clebsch2} must be satisfied for any $\mu$.
Thus, the partial derivative of $\nabla_u H_c $ with respect to $\mu$ must be identically zero, i.e.,
\begin{align}
\begin{aligned}
0 =&
\nabla_{\mu} [\nabla_u H_c] = 
\nabla_{\mu u}^2 L(x, u_c^*, t) + S_xF \\ =& 
S_x\left\{\left[\nabla_{\lambda} g(x, \lambda + \mu S_x^\top, t) \right]^\top\nabla_{uu}^2 L(x, u_c^*, t) + F \right\} 
\nonumber
\end{aligned}
\end{align}
which implies that $[\nabla_{\lambda} g(x, \lambda + \mu S_x^\top, t) ]^\top\nabla_{uu}^2 L(x, u_c^*, t) + F = 0$, yielding ($\because \nabla_{uu}^2 L(x, u_c^*, t) \succ 0$, so invertible)
\begin{align}
\begin{aligned}
\nabla_{\lambda} g(x, \lambda + \mu S_x^\top, t) = - [\nabla_{uu}^2 L(x, u_c^*, t)]^{-1} F^\top
\label{eq:gradientControlLambda}
\end{aligned}
\end{align}
\cref{eq:gradientControlLambda} becomes handy shortly.

From \cref{eq:constrainedHamiltonianCondition}, $\mu^*$ must satisfy $S^{(1)} = 0$ under $u_c^*$, i.e.,
\begin{align}
S^{(1)} = S_x f_0 + S_xFg(x,\lambda + \mu^* S_x^\top,t) + S_t = 0
\label{eq:S1=0condition}
\end{align}
Differentiating $S^{(1)}$ with respect to $\mu$ and using \cref{eq:gradientControlLambda},
\begin{align}
\begin{aligned}
\nabla_{\mu} S^{(1)} =& S_xF [\nabla_\mu g(x,\lambda + \mu S_x^\top,t)] \\
=& S_x F [\nabla_\lambda g(x,\lambda + \mu S_x^\top,t)] S_x^\top \\
=& - S_x F [\nabla_{uu}^2 L(x, u_c^*, t)]^{-1} F^\top S_x^\top
\end{aligned}
\end{align}
which implies $\nabla_{\mu} S^{(1)} <0 $ because $[\nabla_{uu}^2 L(x, u_c^*, t)]^{-1} \succ 0 $ due to \cref{eq:Legendre-Clebsch2} and $S_xF\neq0 $, $\forall x,t$ as discussed in \cref{sec:problem}.
Thus, $S^{(1)} $ is monotonically decreasing in $\mu$.
Note also that, when $\mu=0$, we have that $g(x,\lambda + \mu S_x^\top,t) = g(x,\lambda,t) = u_o^* $ and hence $S^{(1)} = S^{(1)}_o\geq 0 $ on a constrained arc (recall \cref{eq:trajectoryClasification}).
Therefore, $\mu^*$ that satisfies \cref{eq:S1=0condition} is unique and must be non-negative, and the unique solution can be represented by a function $\mu^* = h(x,\lambda,t) $, i.e., \cref{eq:constSolution-mu}.
Finally, $h(x,\lambda,t)$ is continuously differentiable because $S(\cdot)$, $f_0(\cdot)$, $F(\cdot)$, and $g(\cdot)$ are twice differentiable.
\end{proof}

\subsection{Unified Form of Optimality Necessary Conditions}
Thus, we can express the control Hamiltonian and necessary conditions of optimality for \cref{prob:originalOCP} in a unified form without needing to separate the discussion, as follows:
\begin{subequations}
\begin{align}
&H(x,u,\lambda,\mu,t) = L(\cdot) + \lambda^\top f(\cdot) + \mu S^{(1)}(\cdot),
, \label{eq:unifiedHamiltonian}\\
&u^* 					= g(x,\lambda + \mu^* S_x^\top,t)
, \label{eq:unifiedOptimalControl}\\
&\dot{\lambda} 			= - [\nabla_x H(x, u^*, \lambda, \mu^*, t) ]^\top, 
\label{eq:unifiedCostateDynamics} \\
&\mu^* =
\begin{cases}
h(x,\lambda,t) &\mathrm{if}\ S=0 \ \land\ S^{(1)}_o \geq 0 \\
0 & \mathrm{otherwise}
\end{cases}
\label{eq:discontinuousConstraintMultiplier}
\end{align}
\label{eq:unifiedNecessaryCondition}%
\end{subequations}
where \cref{eq:unifiedHamiltonian,eq:unifiedOptimalControl,eq:unifiedCostateDynamics} give the unconstrained arc solution by substituting $\mu^* = 0$;
\cref{eq:discontinuousConstraintMultiplier} is based on \cref{eq:trajectoryClasification}.
This is in sharp contrast to the classical discussion in literature (e.g., \cite{Bryson1975}), which divides a state-constrained optimal control problem into constrained and unconstrained arcs.

\begin{remark}
\label{remark:discontinuity}
Even with the unifying form given by \cref{eq:unifiedNecessaryCondition}, \cref{prob:originalOCP} still experiences discontinuities in $\mu^*$ (and hence in $\lambda^*$ and $H^*$) when the trajectory enters a constrained arc.
\end{remark}

\cref{sec:smoothApprox} addresses the issue identified in \cref{remark:discontinuity}.

\subsection{Analytical Example}
\label{sec:analyticalExample}

Let us consider an example to demonstrate the wide applicability of the assumption made in \cref{sec:assumption} as well as the procedure to find $u^*$ and $\mu^*$ on constrained arcs based on \cref{lemma:uniqueConstrainedSolution}.
This example considers $L(\cdot)$ in the form of
$L(\cdot) = c(x,t) + u^\top Ru + x^\top P u$,
where $R\succ 0 $.
With this, one can find the solution to \cref{eq:unconstOptimalControl} by solving \cref{eq:Legendre-Clebsch} as:
\begin{align}
u_o^* = -\frac{1}{2} R^{-1}(Px + F^\top \lambda)
\label{eq:unconstSolutionExample}
\end{align}
which is continuously differentiable ($\because F$ is continuously differentiable), hence satisfying the assumption made in \cref{sec:assumption}.
It is also straightforward to derive $u^*_c$ by applying \cref{eq:Legendre-Clebsch1,eq:Legendre-Clebsch2}, yielding:
\begin{align}
u_c^* =& -\frac{1}{2} R^{-1}[Px + F^\top (\lambda + \mu^* S_x^\top)]
\label{eq:u-optimal}
\end{align}
which takes the form given in \cref{eq:constSolution}.
Substituting \cref{eq:u-optimal} into \cref{eq:S1=0condition} and solving for $\mu^*$ yields
\begin{align}
\mu^* =& 
\frac{2S_xf_0 + 2S_t - S_x F R^{-1}(Px + F^\top \lambda)}{S_x F R^{-1} F^\top S_x^\top}
\label{eq:mu-optimal}
\end{align}
where $S_x F R^{-1} F^\top S_x^\top \neq0 $ since $S_xF\neq 0 $ and $R \succ 0$.
\cref{eq:mu-optimal} is continuously differentiable because $S(\cdot)$, $f_0(\cdot)$, and $F(\cdot)$ are twice differentiable.


\section{Solution Method}
\label{sec:smoothApprox}

This section presents the proposed solution method to address discontinuities in \cref{prob:originalOCP} and analyzes its properties.

\subsection{Smooth Approximation of Discontinuous Multiplier}
\label{sec:smoothApproxMultiplier}

Let us address the discontinuity issue stated in \cref{remark:discontinuity}.
As clear from \cref{eq:discontinuousConstraintMultiplier}, the discontinuity is triggered by the two conditions $S=0$ and $S^{(1)}_o  \geq 0$.
Hence, this study models $\mu^*$ via $\tilde{\mu}^*$ with two smooth activation functions as:
\begin{align} \begin{aligned}
\tilde{\mu}^*     &=  h(x,\lambda,t) \phi_1(S) \phi_2(S^{(1)}_o),
\label{eq:muApprox}
\end{aligned} \end{align}
where $\phi_1(\cdot) $ and $\phi_2(\cdot) $ aim to smoothly characterize the activation via $S=0$ and $S^{(1)}_o  \geq 0$, respectively.

There are three guiding principles in designing $\phi_1 $ and $\phi_2 $:
(i) for constraint satisfaction, both $\phi_1 $ and $\phi_2$ should take unity when $S=0$ and $S^{(1)}_o \geq 0$;
(ii) $\phi_2$ should remain unity for any values of $S^{(1)}_o \geq 0$ since $S^{(1)}_o$ can be arbitrarily large on constrained arcs;
(iii) for the approximation to be asymptotically precise, $\phi_1$ should be close to zero when $S<0 $ while $\phi_2$ should be close to zero when $S^{(1)}_o < 0$.

Based on these guiding principles, we choose a hyperbolic tangent-based function (a variant of the logistic function) and a variant of the bump function to design $\phi_1 $ and $\phi_2 $, respectively.
Formally defining these activation functions,
\begin{subequations}
\label{eq:activationFunctions}
\begin{align}
\phi_1(x) =& \frac{e^2+1}{2e^2}\left(1 + \tanh{\frac{x + \rho_1}{\rho_1}} \right),
\label{eq:tanhActivation}
\\
\phi_2(x) =& 
\begin{cases}
1               & x > 0 \\
\exp{\left[\frac{1}{\rho_2^2} \left(1 - \frac{1}{1 - x^2} \right) \right]} & x \in (-1, 0]  \\
0               & x \leq -1 \\
\end{cases}
\label{eq:bumpActivation}
\end{align}
\end{subequations}
where $\rho_1\in\R$ and $\rho_2\in\R$ are positive scalars that control the sharpness of the approximation; 
smaller $\rho_1$ and $\rho_2$ lead to sharper approximations.
\cref{f:smoothApprox} illustrates their behavior with different values of $\rho$, where $\phi_{\tanh},\phi_{\mathrm{bump}}$ represent $\phi_1 , \phi_2$.

\cref{theorem:smoothApprox} formally states three key properties the proposed approach possesses.
The first property (1) is crucial to facilitating the numerical convergence of the solution method.
This property implies the smoothness of $u $ and $\dot{x}$, eliminating the need to divide the trajectory into sub-arcs.
The second property (2) is vital when applying to safety-critical systems where we do not want to compromise the state constraint satisfaction for any $\rho$.
The third property (3) ensures the convergence to the local optimum of the original problem.

\begin{lemma}
\label{theorem:smoothApprox}
The smooth approximation $\tilde{\mu}^*$ as defined in \cref{eq:muApprox,eq:activationFunctions} has the following properties:
\begin{enumerate}
    \item $\tilde{\mu}^*$ is continuously differentiable over an optimal trajectory including at corners;
    \item $\tilde{\mu}^*$ respects the inequality constraint regardless of the values of $\rho_1$ and $\rho_2$; and 
    \item $\tilde{\mu}^*$ approaches $\mu^*$ defined in \cref{eq:discontinuousConstraintMultiplier} as $\rho_1,\rho_2\to {+}0$.
\end{enumerate}
\end{lemma}
\begin{proof}
(1) is shown first.
It is clear from \cref{eq:activationFunctions} that $\phi_1 $ and $\phi_2 $ are continuously differentiable in $S$ and $h$, respectively, including at corners.
$S$ and $h$ are both continuously differentiable as discussed earlier, thus proving (1).

For (2), it is clear from \cref{eq:activationFunctions} that, when $S =0 \land S^{(1)}_o >0$, $\phi_1(S)$ and $\phi_2(S^{(1)}_o)$ take both unity independent of the values of $\rho_1$ and $\rho_2$, and hence $\tilde{\mu}^* = h $, yielding $S^{(1)}=0 $, which ensures the constraint satisfaction on constrained arcs.

(3) is verified by showing the following facts:
(a) $\lim_{\rho_1\to{+}0} \phi_1(S) = 0 $ if $S<0 $;
(b) $\phi_1(S) = 1 $ if $S=0 $ independent from $\rho_1$;
(c) $\phi_2(S^{(1)}_o) = 1 $ if $S^{(1)}_o\geq 0 $ independent from $\rho_2$;
(d) $\lim_{\rho_2\to{+}0} \phi_2(S^{(1)}_o)=0$ if $S^{(1)}_o < 0$;
(e) $h = 0$ if $S^{(1)}_o = 0$.
(a) and (b) are shown by noting that $\tanh{(-\infty)}=-1$ and that $ \frac{e^2+1}{2e^2} = \frac{1}{1 + \tanh{1}}$, respectively.
(c) is true by definition, see \cref{eq:bumpActivation}.
(d) is shown by noting that $\lim_{\rho_2\to{+}0} \phi_2(S^{(1)}_o)=0$ if $S^{(1)}_o < 0$ ($\because \exp{(-\infty)} = 0 $).
(e) is shown by noting that the unique $\mu^*$ that simultaneously satisfies $S^{(1)}_o=0 $ and $S^{(1)}=0 $ is necessarily zero, i.e., $\mu^* = h = 0$ ($\because S^{(1)}_o = S_x f_0 + S_xFg(x,\lambda,t) + S_t,\ S^{(1)} = S_x f_0 + S_xFg(x,\lambda + \mu^* S_x^\top,t) + S_t$).
\end{proof}

\begin{figure}[tb] 
    \centering
    \subfigure[\label{f:smoothApprox_1} $\mathrm{tanh}(\cdot)$ activation, \cref{eq:tanhActivation}. The same legend applies to (b) and (c).]{
        \includegraphics[width=\linewidth]{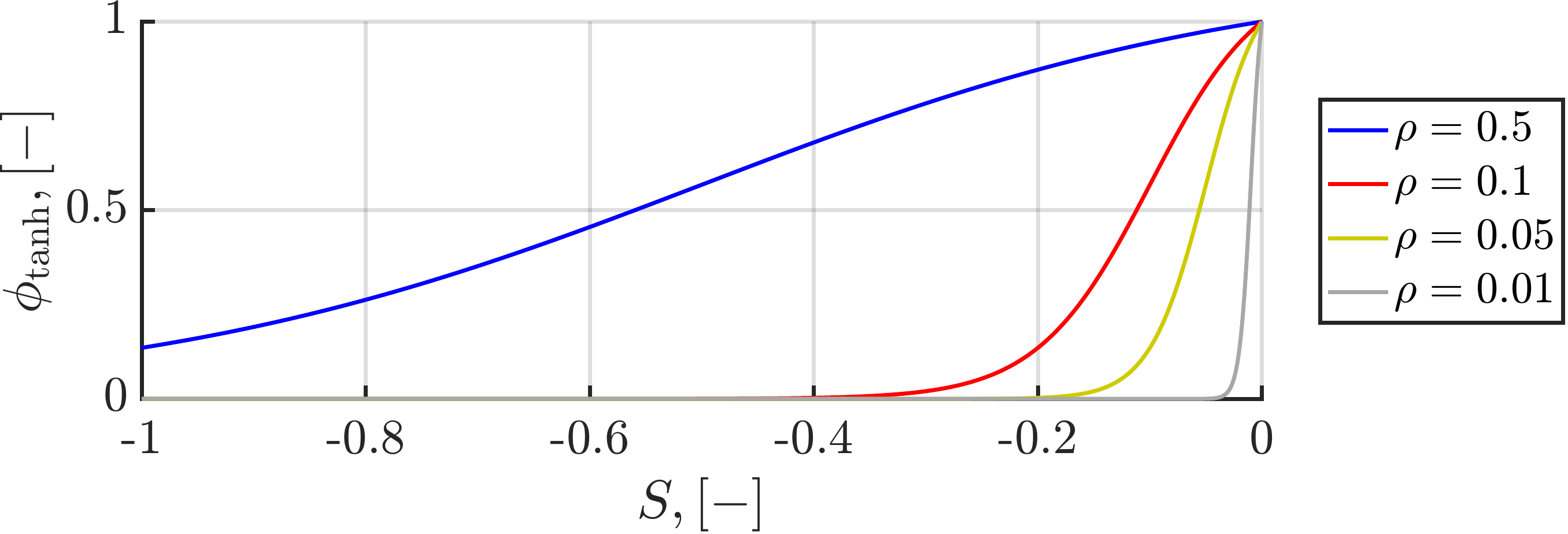}}
    \subfigure[\label{f:smoothApprox_2} Bump function, \cref{eq:bumpActivation}]{
        \includegraphics[width=0.475\linewidth]{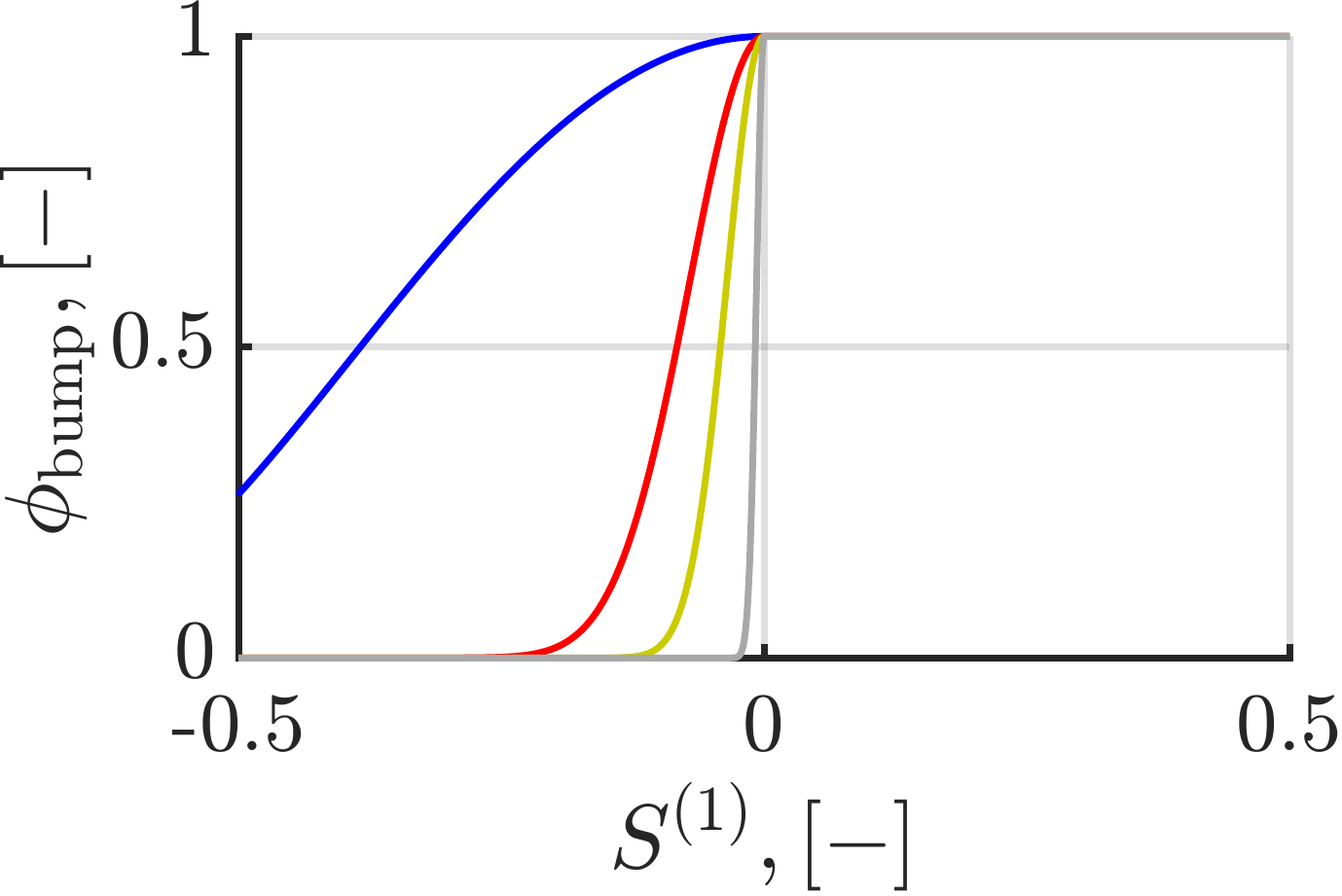}}
    \subfigure[\label{f:diracApprox} Costate jump approx, \cref{eq:approxDirac}]{
        \includegraphics[width=0.475\linewidth]{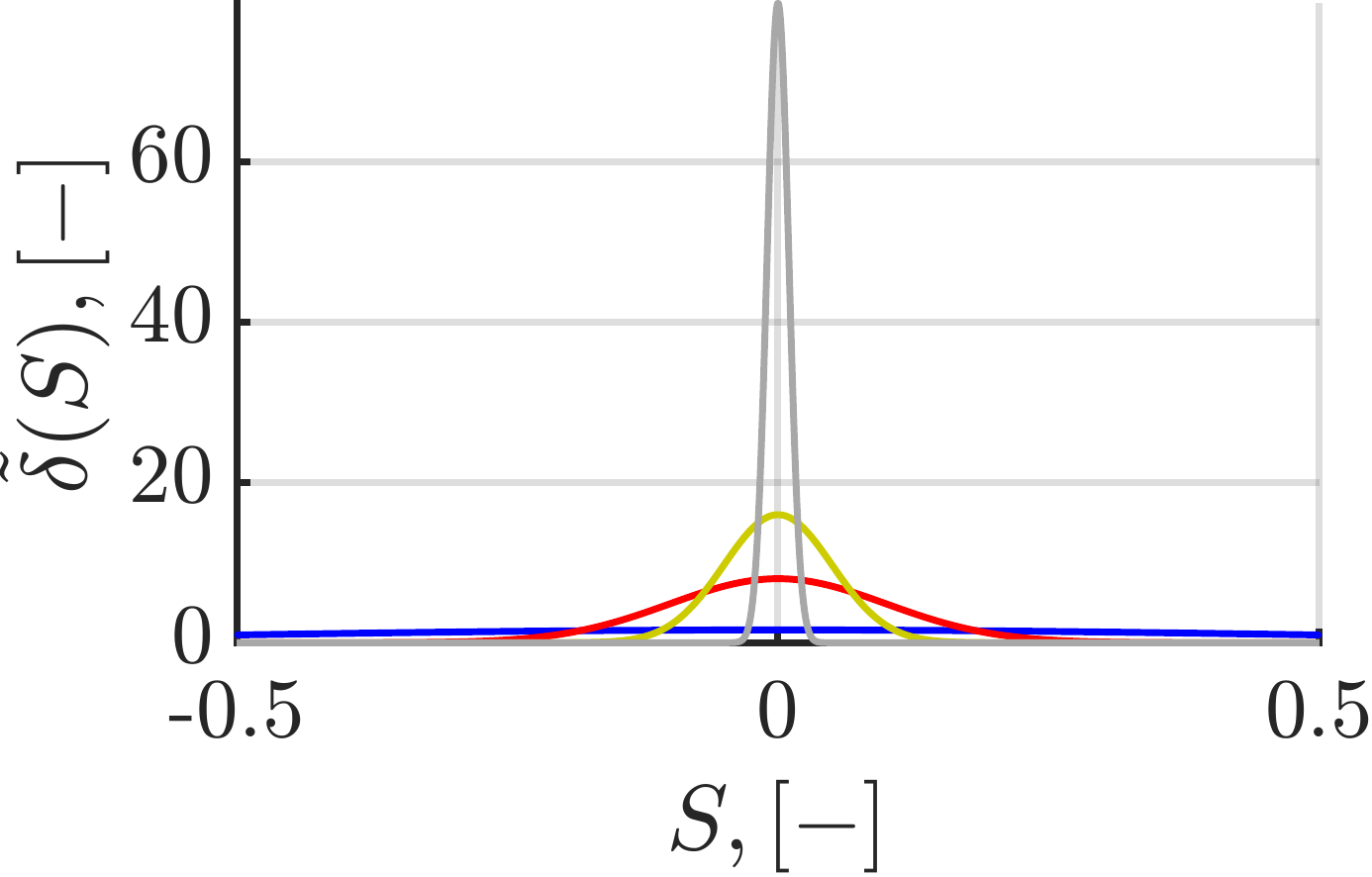}}
    \vspace{-10pt}
    \caption{\label{f:smoothApprox} Activation functions with various sharpness parameters $\rho$ 
    }
    \vspace{-15pt}
\end{figure}

\subsection{Smooth Approximation of Discontinuous Control}

With the smooth multiplier established in \cref{sec:smoothApproxMultiplier}, it is straightforward to approximate $u^*$ via a smooth function:
\begin{align} \begin{aligned}
\tilde{u}^* = g(x, \lambda + \tilde{\mu}^*S_x^\top , t),
\label{eq:controlApprox}
\end{aligned} \end{align}
which has the following desired properties due to \cref{theorem:smoothApprox}:
(1) $\tilde{u}^*$ is continuously differentiable;
(2) $\tilde{u}^*$ respects the inequality constraint regardless of the values of $\rho_1$ and $\rho_2$;
(3) $\tilde{u}^*$ asymptotically approaches $u^*$ as $\rho_1,\rho_2\to +0$.

\subsection{Smooth Approximation of Jump in Costate \& Hamiltonian}

Recalling \cref{eq:costateJump}, the discontinuous change in costate at the $i$-th corner, denoted by $\Delta \lambda_i $, is given by
\begin{align}
\Delta \lambda_i  
= -\mu^*(t_i^+) [S_x(t_i)]^\top 
= -h(x(t_i), \lambda(t_i), t_i) [S_x(t_i)]^\top 
\nonumber
\end{align}
where $\mu^*(t_i^+)$ is replaced by $h(\cdot)$ due to \cref{eq:discontinuousConstraintMultiplier}.
The following formulation directly applies to $\Delta H_i$ as well.

By using the Dirac delta function, $\Delta \lambda_i $ is equivalent to
\begin{align}
\Delta \lambda_i &
= - \int_{-\infty}^{\infty}h(x,\lambda,t) [S_x(t)]^\top \delta(t - t_i)\diff t
\label{eq:jumpCondition}
\end{align}
where $\delta(x)$ satisfies the following conditions:
\begin{align}
&\delta(x) = 
\begin{cases}
\infty & x = 0 \\
0 & x \neq 0
\end{cases}
\label{eq:diracProperty}
,\quad
\int_{-\infty}^{\infty} \delta(x) \diff x = 1
\end{align}
Noting that $t=t_i$ corresponds to a time when $S=0$ and that $S$ never becomes positive under $u^* $ with $\mu^*$, the integral of the Dirac delta can be expressed in terms of $S$ as:
\begin{align} \begin{aligned}
\int_{-\infty}^{\infty}\delta(t - t_i)\diff t
=
2\int_{-\infty}^{0}\delta(S) \diff S
=
2\int_{-\infty}^{t_i}\delta(S(t)) S^{(1)} \diff t
\nonumber
\end{aligned} \end{align}
where $\diff S = \dot{S}\diff t = S^{(1)}\diff t $. 

Obviously, $\delta(x)$ is discontinuous at $t_i$ and needs a smooth approximation for integrating the costate differential equation across $t_i$.
Thus, we smoothly approximate $\delta(x)$ as follows:
\begin{align}
\tilde{\delta}(x) \triangleq
\frac{1}{\rho_3 \sqrt{2\pi} } \exp{\left(-\frac{1}{2} \frac{x^2}{\rho_3^2} \right) },
\label{eq:approxDirac}
\end{align}
which corresponds to the probability distribution function (pdf) of a zero-mean normal distribution with standard deviation $\rho_3$, where $\rho_3$ serves as another sharpness parameter.

It is straightforward to verify that $\tilde{\delta}(x) $ given by \cref{eq:approxDirac} has the following desirable properties:
$\tilde{\delta}(x)$ asymptotically satisfies the left of \cref{eq:diracProperty} as $\rho_3 \to 0$; and
$\tilde{\delta}(x)$ always satisfies the right of \cref{eq:diracProperty} regardless of the value of $\rho_3$, since $\tilde{\delta}(\cdot)$ represents a pdf. 
\cref{f:diracApprox} illustrates the behavior of $ \tilde{\delta}(S)$ with respect to various $\rho_3$.

Using \cref{eq:approxDirac}, the jump condition \cref{eq:jumpCondition} is smoothly approximated via $\widetilde{\Delta \lambda}_i$, calculated as:
\begin{align} \begin{aligned}
\widetilde{\Delta \lambda}_i = - 2 \int_{-\infty}^{t_i}h(x,\lambda,t) [S_x(t)]^\top \tilde{\delta}(S(t)) S^{(1)} \diff t.
\label{eq:jumpConditionApprox}
\end{aligned} \end{align}
Since \cref{eq:jumpConditionApprox} must hold for each $i$-th corner and takes the form of integral with respect to $t$, this effect can be incorporated into \cref{eq:unifiedCostateDynamics}, yielding:
\begin{align} \begin{aligned}
\dot{{\lambda}} =& 
- [\nabla_x H(x, \tilde{u}^*, {\lambda}, \tilde{\mu}^*, t) ]^\top - 2h(\cdot) \tilde{\delta}(S) S^{(1)} S_x^\top
\label{eq:costateDynamicsWithJumps}
\end{aligned} \end{align}
The jump in $H$ can be similarly expressed with a substitution of $S_t$ into $S_x$, resulting an ordinary differential equation for $H$ that takes into account the jump conditions at corners.

Note that the approximation quality of $\Delta \lambda_i$ via \cref{eq:costateDynamicsWithJumps} does not affect the feasibility of the problem.
The feasibility is guaranteed by the second property of \cref{theorem:smoothApprox}.

\subsection{Smooth State-constrained Optimal Control Problem}

Applying the proposed solution method transforms \cref{prob:originalOCP} into \cref{prob:smoothOCP}, which no longer involves discontinuities in control and costate and represents a smooth TPBVP.

\begin{problem}[Smooth State-constrained Problem]
\label{prob:smoothOCP}
For given sharpness parameters $\{\rho_1,\rho_2,\rho_3\} $, find the initial costate $\lambda(0) $ (and $T$ for variable-time problems) such that satisfy the transversality conditions subject to the state dynamics \cref{eq:dynamics} and costate dynamics \cref{eq:costateDynamicsWithJumps} under control \cref{eq:controlApprox}.
\end{problem}

Due to the smoothness, we can solve \cref{prob:smoothOCP} by 
the standard indirect shooting with no need for dividing the problem into sub-arcs.
For completeness, the standard indirect shooting procedure can be described as follows:
(1) define an augmented state, $X = [x^\top, {\lambda}^\top]^\top$, which obeys:
\begin{align}
\hspace{-1pt}
\dot{X} = 
\begin{bmatrix}
f(x, \tilde{u}^*, t) \\
- [\nabla_x H(x, \tilde{u}^*, {\lambda}, \tilde{\mu}^*, t) ]^\top - 2\tilde{\mu}^* \tilde{\delta}(S) S^{(1)} S_x^\top
\end{bmatrix};
\label{eq:augODE}
\end{align}
(2) set up a shooting function $\Psi(\lambda(0))$ which, for given $\lambda(0)$, integrates \cref{eq:augODE} from $t=0$ to $T$ and evaluates the transversality conditions;
(3) use a nonlinear root-finding solver to find $\lambda(0)$ that yields $\Psi(\cdot) = 0$.

To approach the optimal solution of the original problem, we perform continuation over a decreasing sequence of $\rho$'s, where $\rho\triangleq\rho_1=\rho_2=\rho_3$ is used for simple algorithm design.
That is, after solving \cref{prob:smoothOCP} with $\rho^{(i)}$, the solution is used as the initial guess of $\lambda(0)$ for the next iteration with $\rho^{(i+1)}$, where $\rho^{(i)}$ decreases as $i$ increases.
This approach takes advantage of the third property in \cref{theorem:smoothApprox} and the convergence property of $\tilde{\delta}(\cdot) $ to $\delta(\cdot)$.

\begin{remark}
\label{remark:constraintSatisfaction}
Each intermediate solution of the continuation procedure respects the state constraint regardless of the values of the sharpness parameters $\{\rho_1,\rho_2,\rho_3\} $ (i.e., anytime algorithm) due to the second property of \cref{theorem:smoothApprox}, which is often a desirable property for safety-critical systems.
\end{remark}



\section{Numerical Example}

We consider a 2-D orbit transfer problem.
Let $r,v\in\R^2 $ be the position and velocity of the spacecraft, $x \in \R^4 $ be the state vector, and $u\in\R^2 $ the control vector, defined as:
\begin{align}
x = 
\begin{bmatrix}
r \\
v
\end{bmatrix}
,\quad
r =
\begin{bmatrix}
r_1 \\
r_2
\end{bmatrix}
,\quad
v =
\begin{bmatrix}
v_1 \\
v_2
\end{bmatrix}
,\quad
u =
\begin{bmatrix}
u_1 \\
u_2
\end{bmatrix}
\end{align}
Orbital dynamics with gravity parameter $\mu_g$ are given by:
\begin{align}
f(x,u) =
f_0(x) + Bu,
\ 
f_0 =
\begin{bmatrix}
v \\
-\frac{\mu_g}{\norm{r}_2^3} r 
\end{bmatrix}
,\ 
B =
\begin{bmatrix}
0_{2\times2} \\
I_2
\end{bmatrix}
\nonumber
\end{align}
where the canonical unit is used so that $\norm{r(t=0)}_2 = \mu_g = 1$.
$\norm{\cdot}_2$ represents the 2-norm of a vector.
The cost function, boundary conditions, and state constraint are defined as:
\begin{align} \begin{aligned}
L = \frac{1}{2} \norm{u}_2^2
,\ 
x(0) = x_0
,\ 
x(T) = x_T
,\ 
S = p_{\min} - p(x),
\nonumber
\end{aligned} \end{align}
where
$p(x) = (r_1 v_2 - r_2 v_1)^2 / \mu_g$ is the semilatus rectum of an orbit.
$x_0, x_T, T, p_{\min}$ are given and defined as:
$x_0 = [1, 0, 0, 1]^\top,\ x_T = [3\sqrt{3}/2, 1.5, - 1/(2\sqrt{3}), 0.5]^\top,\ T = 3\pi,\ p_{\min} = 0.9 $.
The necessary conditions are derived as:
\begin{align} \begin{aligned}
u^* = - B^\top (\lambda + \mu^* S_x^\top),
\quad
\mu^* = \frac{S_x f_0 - S_x B B^\top \lambda}{S_x B B^\top S_x^\top}
\end{aligned} \end{align}
Note that this problem is a special case of the example in \cref{sec:analyticalExample} with $F = B, S_t = 0, c = 0, R = \frac{1}{2} I_2 $, and $P=0$.
The transversality condition is $x(T) - x_T = 0 $.

\begin{figure}[tb] 
    \centering
    \subfigure[\label{f:ex1-1} Optimal trajectories for various $\rho$. The same legend applies to (b)-(e).
    The filled circle and triangle represent the initial and final positions.]{
        \includegraphics[width=\linewidth]{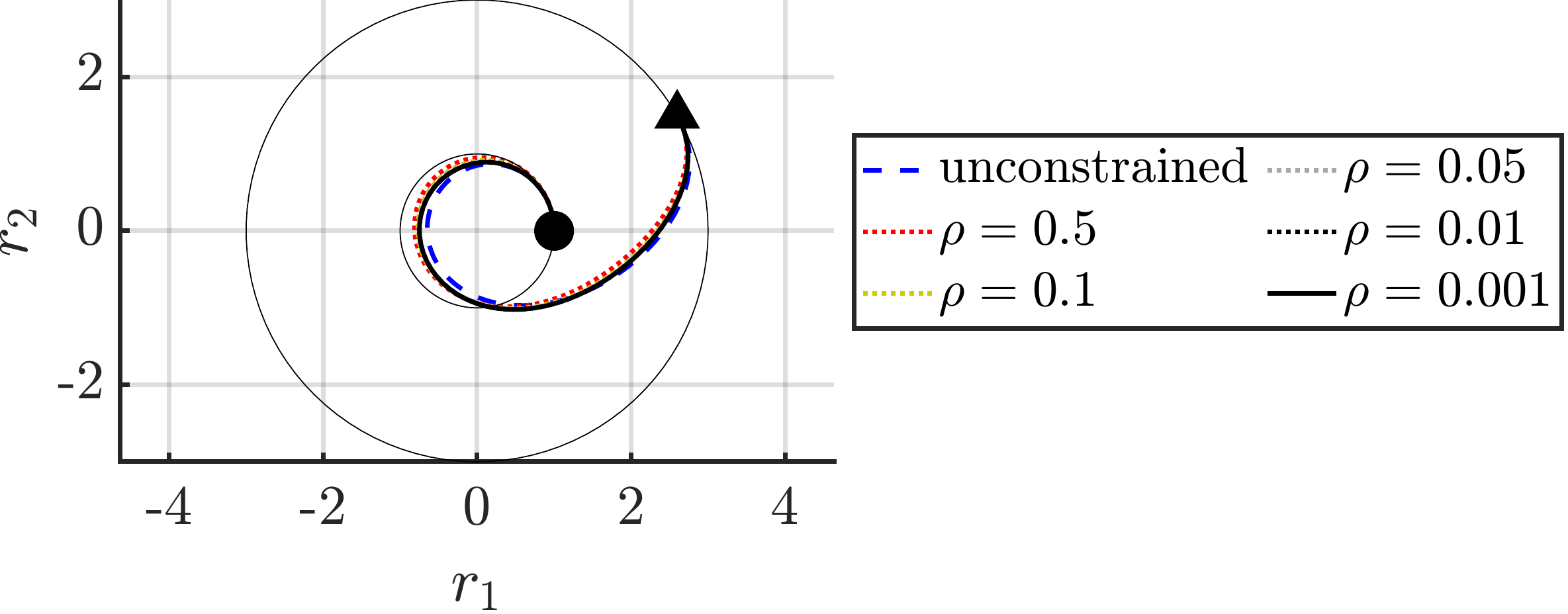}}
    \subfigure[\label{f:ex1-2} Optimal control profiles]{
        \includegraphics[width=0.475\linewidth]{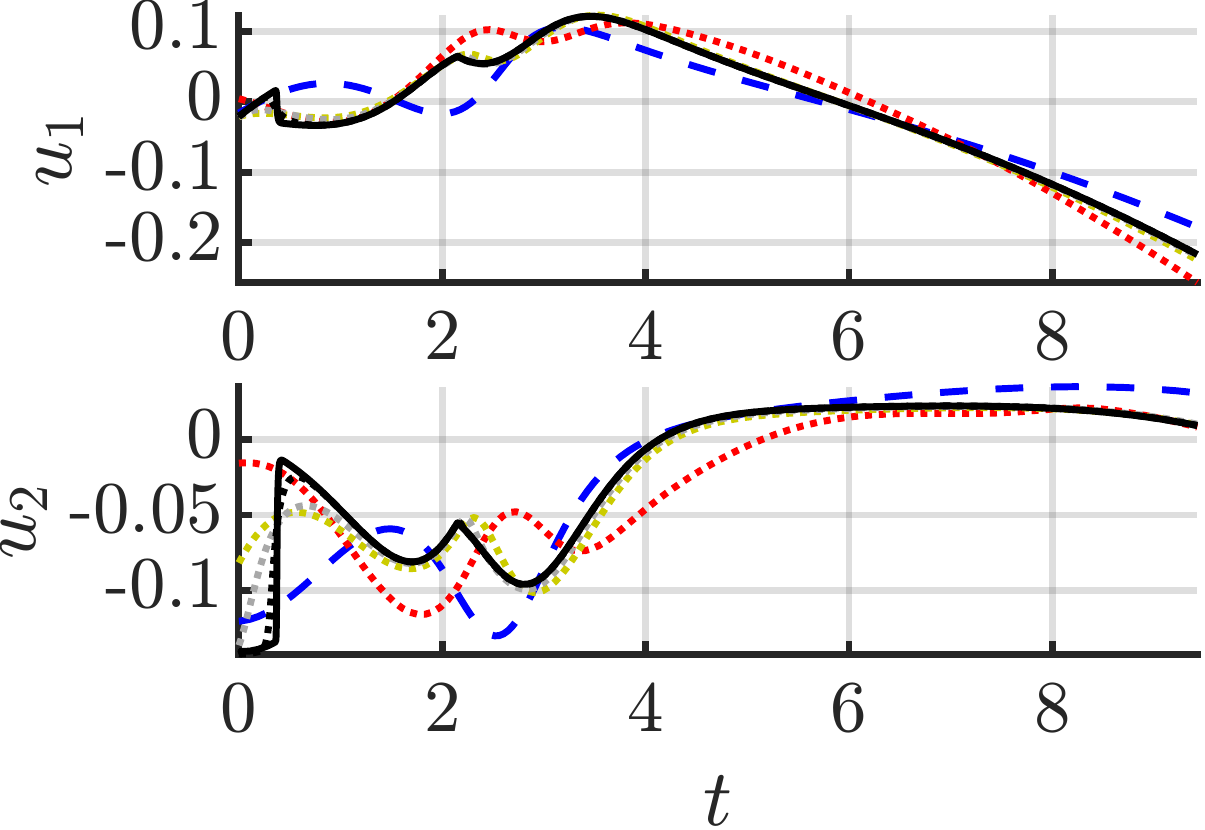}}
    \subfigure[\label{f:ex1-3} Optimal costate profiles]{
        \includegraphics[width=0.475\linewidth]{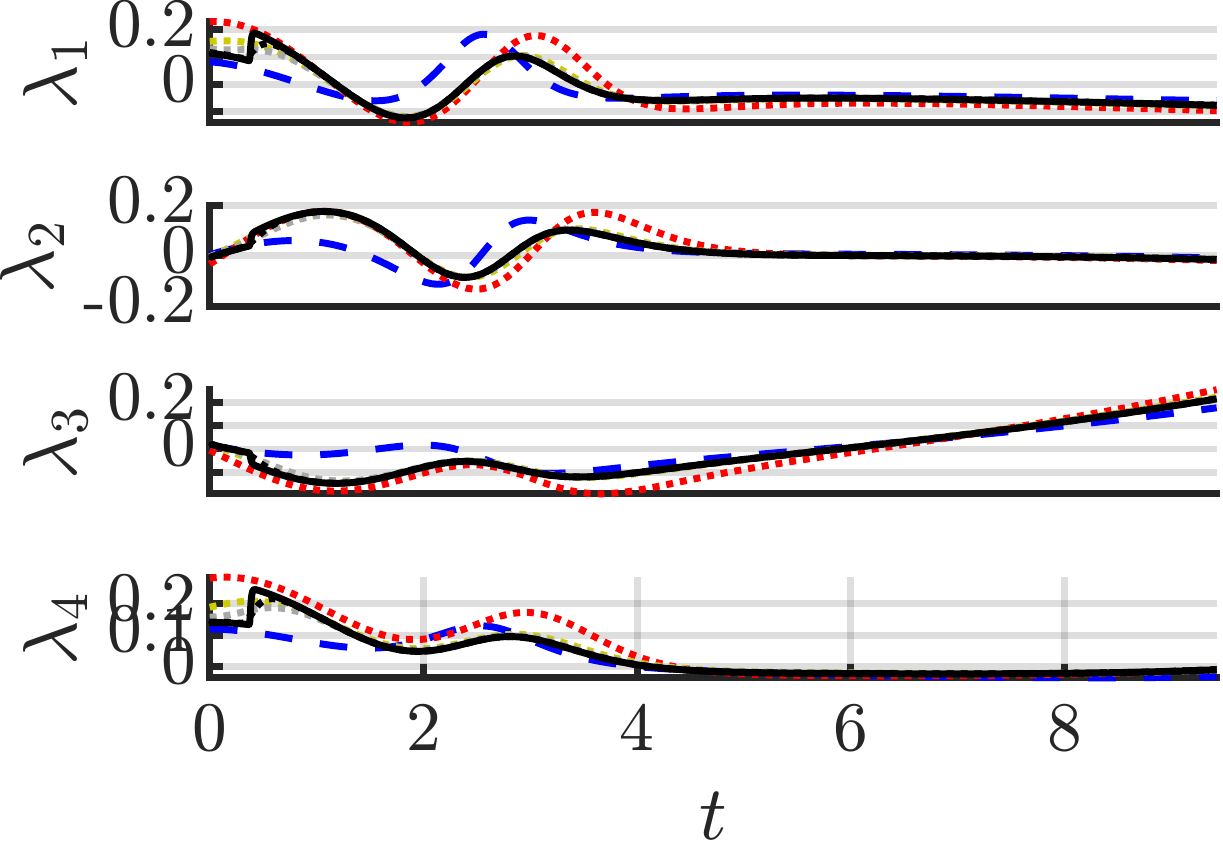}}
    \subfigure[\label{f:ex1-4} Smoothed $\mu$ and costate jump]{
        \includegraphics[width=0.475\linewidth]{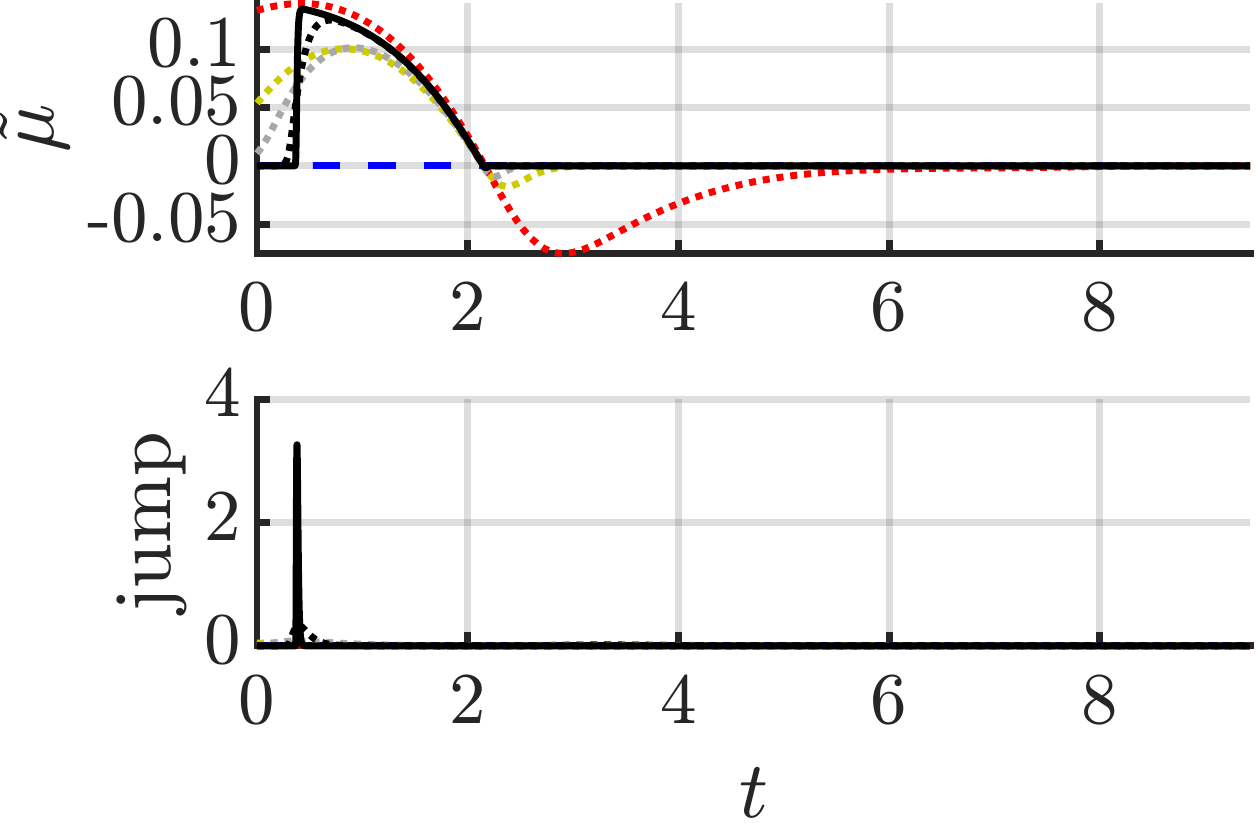}}
    \subfigure[\label{f:ex1-5} State constraint and derivative]{
        \includegraphics[width=0.475\linewidth]{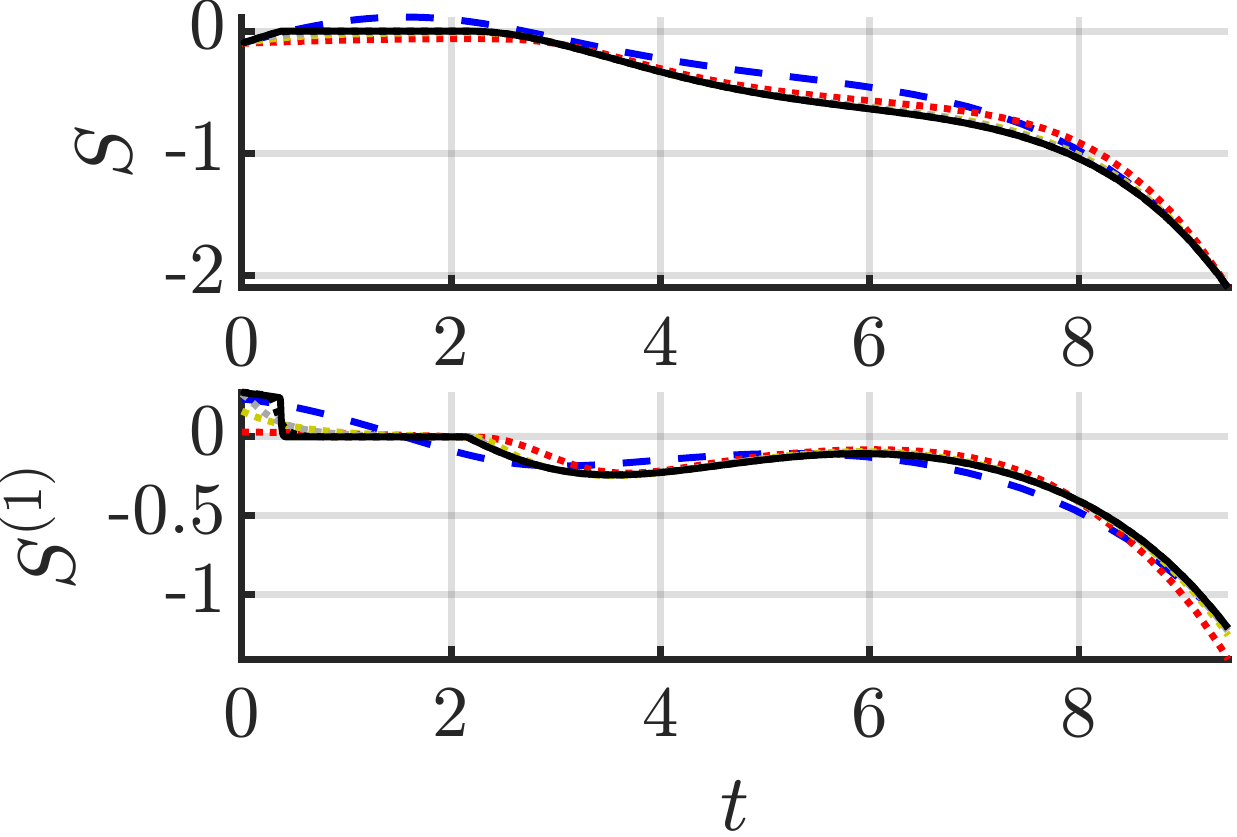}}
    \vspace{-10pt}
    \caption{\label{f:Ex1-1} Numerical example results with various $\rho$.
    }
    \vspace{-15pt}
\end{figure}

We solve this problem in the form of \cref{prob:smoothOCP} by using Matlab's \texttt{fsolve},
where \texttt{ode45} is used for the integration of \cref{eq:augODE}.
The tolerances used are 
$\mathtt{Opt/FunTol}=10^{-8}$ for \texttt{fsolve}
and
$\mathtt{Rel/AbsTol}=10^{-8}$ for \texttt{ode45}.

The solutions for $\rho^{(i)} \in \{0.5, 0.1, 5\times 10^{-2}, 10^{-2}, 5\times 10^{-3}, 10^{-3} \}$ are shown in \cref{f:Ex1-1}, including the unconstrained solution for comparison.
\cref{f:ex1-2,f:ex1-3} indicate that control and costate experience a discontinuity around $t=0.55$, which is well-modeled by the smooth approximation.
\cref{f:ex1-2} shows a significant difference in the optimal control profiles for the constrained and unconstrained problems.
\cref{f:ex1-4} shows the behavior of $\tilde{\mu}^*$ (constraint multiplier) and $2h\tilde{\delta}(S) S^{(1)} $  (time derivative of costate jump; see \cref{eq:costateDynamicsWithJumps}), which numerically verifies the analysis in \cref{sec:smoothApprox}.
\cref{f:ex1-5} plots $S$ and $S^{(1)} $ over $t$, confirming 
that intermediate solutions in the continuation respect the state constraint, 
even for a blunt sharpness parameter $\rho=0.5$, which demonstrates \cref{remark:constraintSatisfaction}.

\section{Conclusions}
In this paper, a new indirect solution method for state-constrained optimal control problems is presented to address the long-standing issue of discontinuous control and costate due to state inequality constraints.
The proposed solution method is enabled by re-examining the necessary conditions of optimality for the constrained control problems and deriving a unifying form of necessary conditions based on the uniqueness of the optimal control and constraint multiplier on constrained arcs.
In contrast to classical indirect solution methods, the proposed method transforms the originally discontinuous problems into smooth TPBVPs, eliminating the need for \textit{a priori} knowledge about the optimal solution structure, which is usually unknown.
A numerical example demonstrates the proposed method and its anytime algorithm-like property.
A key next direction of this work is to generalize the framework to higher-order state constraints.


\bibliographystyle{ieeetr}
\bibliography{../../../../../../utility/latex/ref/zotero/library}

\end{document}